\documentclass{article}

\usepackage[utf8]{inputenc}
\usepackage{amsfonts}
\usepackage{amsmath,amsthm,amssymb}
\usepackage{mathtools}
\usepackage[colorlinks]{hyperref}
\usepackage[nameinlink,capitalize]{cleveref}
\usepackage{pgfplots}
\usepackage{enumitem}

\usepackage{algorithm}
\usepackage{algorithmicx}
\usepackage{algpseudocode}
\usepackage{verbatim}
\usepackage{subcaption}
\usepackage{authblk}

\newtheorem{theorem}{Theorem}[section]
\newtheorem{proposition}[theorem]{Proposition}
\newtheorem{lemma}[theorem]{Lemma}
\newtheorem{corollary}[theorem]{Corollary}

\theoremstyle{definition}
\newtheorem{definition}[theorem]{Definition}
\newtheorem{assumption}[theorem]{Assumption}

\theoremstyle{remark}
\newtheorem{remark}[theorem]{Remark}

\DeclareMathOperator*{\minimize}{minimize}

\newcommand{\N}{\mathbb{N}}
\newcommand{\R}{\mathbb{R}}
\newcommand{\Z}{\mathbb{Z}}
\newcommand{\Ha}{\mathcal{H}}
\newcommand{\conv}{\operatorname{conv}}

\newcommand{\st}{\operatorname{subject\ to}}
\DeclareMathOperator*{\TV}{TV}

\DeclareMathOperator*{\pred}{pred}
\DeclareMathOperator*{\ared}{ared}

\DeclareMathOperator*{\BV}{BV}

\DeclareMathOperator*{\dvg}{div}

\newcommand*\dd{\mathop{}\!\mathrm{d}}

\newcommand{\weakto}{\rightharpoonup}
\newcommand{\pweakstarto}{\stackrel{p\ast}{\rightharpoonup}}

\newcommand{\calF}{\mathcal{F}}

\DeclarePairedDelimiterX\set[1]\{\}{#1}

\DeclarePairedDelimiterX\innerp[2](){#1,#2}
\DeclarePairedDelimiterX\dual[2]{\langle}{\rangle}{#1,#2}

\definecolor{darkgreen}{rgb}{0,0.5,0}

\newcommand{\symdiff}{\mathbin{\triangle}}

\title{Integer Optimal Control with Fractional Perimeter Regularization}
\author[1]{Harbir Antil \footnote{H. Antil is partially supported by NSF grant DMS-2110263, Air Force Office of Scientific Research (AFOSR) under Award NO: FA9550-22-1-0248, and Office of Naval Research (ONR) under Award NO: N00014-24-1-2147.}}
\author[2]{Paul Manns\footnote{Paul Manns acknowledges funding by Deutsche Forschungsgemeinschaft (DFG) under project no.\ 515118017.}}
\affil[1]{Department of Mathematical Sciences and Center for Mathematics and Artificial Intelligence, George Mason University, Fairfax VA 22030, USA, \textit{hantil@gmu.edu}}
\affil[2]{Faculty of Mathematics, TU Dortmund University, 44227 Dortmund, Germany, \textit{paul.manns@tu-dortmund.de}}

\begin{document}
\maketitle
\begin{abstract}
Motivated by many applications, optimal control problems with integer
controls have recently received a significant attention. Some state-of-the-art work uses perimeter-regularization to derive stationarity conditions and trust-region algorithms. However, the discretization is difficult in this case because the perimeter is concentrated on a set of dimension $d - 1$ for a domain of dimension $d$.

This article proposes a potential way to overcome this challenge by using the fractional nonlocal perimeter with 
fractional exponent $0<\alpha<1$. In this way, the boundary integrals in the perimeter regularization are replaced by
volume integrals. Besides establishing some non-trivial properties associated with this perimeter, a $\Gamma$-convergence result is derived. This result establishes convergence of minimizers of fractional perimeter-regularized problem, to the standard one, as the exponent $\alpha$ tends to 1. In addition, the stationarity results are derived and algorithmic convergence analysis is carried out for $\alpha \in (0.5,1)$ under an additional assumption on the gradient of the reduced objective. 

The theoretical results are supplemented by a preliminary computational experiment. We observe that the isotropy of the total variation may
be approximated by means of the fractional perimeter functional.
\end{abstract}

\section{Introduction}

Let $\Omega \subset \R^d$, $d \in \N$, be a bounded, polyhedral
domain.
Let $M \in \N$ and $\eta > 0$. This article is concerned
with solving minimization problems of the form
\begin{gather}\label{eq:p}
\begin{aligned}
\minimize_{w \in L^1(\Omega)} \enskip & J_\alpha(w) \coloneqq
F(w) + \eta R_\alpha(w)\\
 \st\enskip & w(x) \in W \coloneqq \{w_1,\ldots,w_M\} \subset \Z \text{ a.e.\ in } \Omega.
\end{aligned}
\tag{P$_\alpha$}
\end{gather}
Here, $F : L^1(\Omega) \to \R$, which we assume to be bounded
below, is the principle part of 
the objective that is due to the application and $R_\alpha$ is a 
regularizer that will provide desirable features of the 
solution to \eqref{eq:p}. The scalar $\alpha \in (0,1)$
parameterizes the regularizer
$R_\alpha$ and in turn \eqref{eq:p}. The specific role 
of $\alpha$ will become clear soon.
In order to clarify possible misunderstandings, since the term 
polyhedron may or may not imply convexity in the literature, 
%we note that we use 
the term polyhedron in this paper is used for sets whose 
boundaries are unions of convex polytopes.

Recent work
\cite{leyffer2022sequential,manns2023on,marko2022integer,severitt2023efficient} motivates and analyzes
the use of a total variation regularization of $w$,
which corresponds to a penalization of the perimeters of the level sets because 
\[
\TV(w) \le \sum_{i=1}^M|w_i| P(w^{-1}(\{w_i\});\R^d) \le 2 \max_i |w_i| \TV(w)
\]
if $\TV(w) < \infty$, where $P(A;B)$ denotes the perimeter of $A$ in
$B$ and $\TV(w)$ denotes the total variation of the function $w$, where we assume that $w$ is
extended by outside of $\Omega$ for feasible $w$ in \eqref{eq:p}, see \cite[Lemma 2.1]{manns2023on}.

We refer the reader to \cite{maggi2012sets}
for extensive information on sets of finite perimeter and
their properties. The key property of this regularization
that is exploited in the aforementioned publications is
the compactness it induces on the control space, specifically
bounded sequences of feasible points of \eqref{eq:p}, like
sequences produced by descent algorithms, have
subsequences that converge in $L^1(\Omega)$.

In the multi-dimensional case, $d \ge 2$, the 
discretization of the subproblems that are proposed in
\cite{leyffer2022sequential,manns2023on} is challenging.
The reason is that the arguments in the analysis of the 
finite difference or finite element discretizations, 
as are for example carried out in \cite{bartels2012total,chambolle2017accelerated,caillaud2023error}, 
use that $w$ can take values in $\R$ (or at least $\conv W$, the convex hull of $W$)
and that the superordinate minimization problem is convex. Both of these features are not available
in our setting. Moreover, any piecewise constant ansatz for $w$ on a fixed
decomposition of the domain into say polytopes restricts
the geometry of the level sets and therefore introduces
a potential gap between \eqref{eq:p} and its discretization.
This gap is due to the local structure of the perimeter
regularization. Specifically, the information on the perimeter
is concentrated on the (reduced) boundary of the level sets
and hence the discretization, which has finite $d-1$-dimensional
Hausdorff measure. The very recent work \cite{schiemann2024discretization}
provides a two-level discretization that can overcome this issue but
is computationally expensive and efficient implementations that make use of the
underlying structure as in \cite{manns2024discrete} are not available so far.

This motivates us to study regularization terms $R_\alpha$ that are close
to the perimeter regularization, provide compactness in $L^1(\Omega)$, but also have
non-local properties so that they might give a fruitful computational vantage point
because they allow to replace the difficult localized boundary
integrals: specifically, $R_\alpha$ is given by means of a double volume integral so that
its numerical approximation can be improved by improving the quadrature of the volume integrals.
Let $w$ be feasible for \eqref{eq:p} and let $E_i \coloneqq w^{-1}(\{w_i\})$ for
$i \in \{1,\ldots,M\}$. Specifically, we consider
\[ R_\alpha(w) \coloneqq (1 - \alpha) \sum_{i=1}^M |w_i|P_\alpha(E_i) \]
for $\alpha \in (0,1)$, where $P_\alpha(E)$ is the
so-called \emph{fractional perimeter}
\begin{gather}\label{eq:fractional_perimeter}
P_\alpha(E) 
\coloneqq \int_{\R^d} \int_{\R^d} \frac{|\chi_E(x) - \chi_E(y)|}{|x - y|^{d + s}} \dd x \dd y,
\end{gather}
introduced and analyzed in \cite{caffarelli2010nonlocal,visintin1991generalized}
and $\chi_E$ denotes the $\{0,1\}$-valued indicator function of $E \subset \R^d$,
see also \cite{comi2019distributional,cozzi2017regularity,HAntil_HDiaz_TJing_ASchikorra_2024a}.
The limit problem that we approximate with \eqref{eq:p} is the perimeter-regularized
integer optimal control problem
\begin{gather}\label{eq:p0}
\begin{aligned}
\minimize_{w \in L^1(\Omega)} \enskip &J(w) \coloneqq
F(w) + \eta R(w)\\
\st\enskip & w(x) \in W \coloneqq \{w_1,\ldots,w_M\} \subset \Z \text{ a.e.\ in } \Omega,
\end{aligned}
\tag{P}
\end{gather}
where $R(w) = \omega_{d-1}\sum_{i=1}^M |w_i|P(E_i;\R^d)$ and $\omega_{d-1}$ denotes the
volume of the unit ball in $\R^{d-1}$. In particular, this means that our optimization
variable $w$ can be modified only on the domain $\Omega$ but the regularizer 
in \eqref{eq:p}, \eqref{eq:p0} takes into account the boundary of $\Omega$. In other words,
$w$ is implicitly extended by zero outside of $\Omega$ and the jumps across $\partial \Omega$
are counted.
We note that we are not the first ones to make steps in the direction of computationally
exploiting these properties of the fractional perimeter and particularly point to the
works \cite{dipierro2016nonlocal,borthagaray2023fractional,HAntil_HDiaz_TJing_ASchikorra_2024a}.

\subsection{Contributions}
We make the following contributions. The existence of solutions to \eqref{eq:p} is---in our opinion---not 
immediate since we are not aware of a Banach--Alaoglu theorem for the Sobolev space $W^{\alpha,1}(\R^d)$
that induces the fractional perimeter as defined in \eqref{eq:fractional_perimeter}.
Thus we prove it by means of an argument that exploits the specific structure of our feasible set that only consists
of $W$-valued functions, $|W| < \infty$.

We show compactness in $L^1(\Omega)$ and $\Gamma$-convergence for $\alpha \nearrow 1$. For all
$\alpha \in (0.5,1)$, we prove stationarity conditions and asymptotics of a trust-region algorithm
parallel to \cite{manns2023on}. In order to achieve this, we currently require the regularity
assumption $\nabla F(w^n) \in C^2(\bar{\Omega})$ for the iterates $w^n$ produced by the algorithm,
which is quite strong compared to the regularity $\nabla F(w^n) \in C(\bar{\Omega})$ that
is required for the perimeter-regularized case, see \cite{manns2023on}.

We also provide a preliminary and qualitative computational experiment, in which we apply
the trust-region algorithm for the choices $\alpha = 0.5$ and $\alpha = 0.9$
as well as for $R(w) = \omega_{d-1}\sum_{i=1}^M |w_i|P(E_i;\R^d)$.
The geometric restriction induced by the piecewise constant ansatz for
our control functions is visible in the limiting case but this behavior
is alleviated for $\alpha = 0.5$ and $\alpha = 0.9$.
Unfortunately, the subproblem solves in the trust-region algorithm
are extremely expensive even for a relatively coarse discretization.
Moreover, since computationally tractable discretizations of the subproblems for the limit case are not available
so far, it is difficult to interpret and compare the results.
Therefore, we emphasize that much more work is needed from a computational point of view
in order to provide more efficient discretization and solution algorithms.

\subsection{Structure of the remainder}
After introducing some notation and the necessary concepts regarding modes of convergence and
local variations of the elements of the feasible set in \cref{sec:notation}, we provide
the existence of solutions in \cref{sec:existence}. Compactness and $\Gamma$-convergence
are analyzed in \cref{sec:gamma}.
The analysis of local minimizers and the trust-region algorithm and its asymptotics
are provided and analyzed in \cref{sec:stat_and_tr}.
We provide a computational experiment and discuss its implications in \cref{sec:computational}.

\section{Notation and auxiliary results}\label{sec:notation}

If not indicated otherwise, we assume that $\alpha \in (0,1)$
is fixed but arbitrary in the whole article without further
mention. We denote the complement of a set $A \subset \R^d$
by $A^c \coloneqq \R^d\setminus A$. We denote the symmetric difference between
$A$ and a further set $B \subset \R^d$ by $A \symdiff B$.
We will frequently use the following reformulation of \eqref{eq:fractional_perimeter}.
Let $E \subset \R^d$, then $P_\alpha(E)$ satisfies
\begin{gather}\label{eq:frac_perim_as_one_double_integral}
P_\alpha(E) 
= 2 \int_E \int_{E^c}
  \frac{1}{|x - y|^{d + \alpha}} \dd x \dd y.
\end{gather}
We immediately obtain that the function $P_\alpha$ is submodular, see also (2.1)
in \cite{chambolle2013minimizing}.
\begin{lemma}\label{lem:submodularity}
Let $\alpha \in (0,1)$. Let $E$, $F$ be measurable subsets of $\Omega$. Then
\begin{gather}\label{eq:union_intersection_inequality}
P_\alpha(E \cap F) + P_\alpha(E \cup F) \le P_\alpha(E) + P_\alpha(F) 
\end{gather}   
and
\begin{gather}\label{eq:set_subtraction_inequality}
P_\alpha(E \setminus F) + P_\alpha(F \setminus E) 
\le P_\alpha(E) + P_\alpha(F) \, .
\end{gather}   
\end{lemma}
\begin{proof}
We consider the formulation of $P_\alpha(E)$ from \eqref{eq:frac_perim_as_one_double_integral} and
define $g(x,y) \coloneqq |x - y|^{-(d + \alpha)}$ for $x$, $y\in \R^d$.
Then inserting the definitions, elementary computations, and the positivity of $g$ yield
\begin{align*}
\frac{1}{2}P_\alpha(E \cap F) +
\frac{1}{2}P_\alpha(E \cup F)
&= \int_{E \cap F} \int_{(E \cap F)^c}g + \int_{E \cup F} \int_{(E \cup F)^c} g \\
&=  \int_{E \cap F} \int_{(E \setminus F) \cup (F \setminus E)}g  + \int_{E} \int_{(E \cup F)^c} g + \int_F \int_{(E \cup F)^c} g  \\
&\le  \int_{E \cap F} \int_{E \setminus F}g  + \int_{E} \int_{E^c} g + \int_F \int_{(E \cup F)^c} g  \\
&\le \int_{E} \int_{E^c} g + \int_F \int_{F^c} g  = \frac{1}{2}P_\alpha(E; \Omega) + \frac{1}{2}P_\alpha(F; \Omega),
\end{align*}
which proves \eqref{eq:union_intersection_inequality}. In order to see
\eqref{eq:set_subtraction_inequality}, we consider
\begin{align*}
\frac{1}{2}P_\alpha(E \setminus F) +
\frac{1}{2}P_\alpha(F \setminus E)
&= \int_{E \setminus F} \int_{(E \setminus F)^c}g 
 + \int_{F \setminus E} \int_{(F \setminus E)^c} g \\
&= \int_E\int_{E^c}g + \int_{E \setminus F}\int_{E\cap F}g - \int_{E\cap F}\int_{E^c} g\\
&\hphantom{=} + \int_F\int_{F^c}g + \int_{F \setminus E}\int_{E \cap F}g - \int_{E \cap F}\int_{F^c} g\\
&\le \frac{1}{2}P_\alpha(E) + \frac{1}{2}P_\alpha(F).
\end{align*}
The proof is complete.
\end{proof}

We will sometimes switch between the view of a function $w$ that
is feasible for \eqref{eq:p} and the partition of $\Omega$
that is given by its level sets. If there is no ambiguity,
we will denote the level sets by $E_i \coloneqq w^{-1}(\{w_i\})$,
$i \in \{1,\ldots,M\}$, without further mention.
Let $\lambda$ denote the Lebesgue measure on $\R^d$.

Let $\alpha \in (0,1)$. Then $\nabla^\alpha f$, defined as
\[ \nabla^\alpha f(x) = 
\int_{\R^d}
\frac{(y - x) \cdot (f(y) - f(x))}{|y - x|^{d + 1 + \alpha}}\dd y
\]
for $x \in \R^d$ is the so-called fractional gradient \cite{comi2019distributional,vsilhavy2020fractional,HAntil_HDiaz_TJing_ASchikorra_2024a}
for all $f$, where the integrand of this integral is an integrable function.

We denote the feasible set by
\[ \calF \coloneqq \{ w \in L^1(\Omega)\,|\, w(x) \in W \text{ for a.e.\ } x \in \Omega \}.
\]

\subsection{Modes of convergence}
The Gagliardo seminorm of the Sobolev space $W^{\alpha,1}(\R^d)$, see \cite{di2012hitchhikers},
with fractional order of differentiability $\alpha \in (0,1)$
corresponds to the fractional perimeter $P_\alpha$ as defined above. To the best of the authors' knowledge, this space
does not admit a predual so that, in contrast to $\BV(\R^d)$, there is no weak-$^*$ topology that gives existence
of limits for bounded subsequences. However, this property can be recovered when restricting to our feasible set
$\calF$ of a.e.\ $W$-valued integrable functions. We therefore refer to this property as pseudo-weakly-$^*$ in
this article. In Banach spaces that are not uniformly convex, that is they do not have a norm that satisfies a uniform midpoint
convexity property, having weak-$^*$ convergence or weak convergence together with convergence
of the values of the norm does not necessarily imply convergence in norm. Consequently, there is sometimes an
important mode of convergence for this situation like so-called strict convergence in $\BV(\Omega)$.
In an analogy to this, we define strict convergence for our setting as convergence in $L^1(\Omega)$ in combination
with convergence of the regularizer, which of course implies pseudo-weak-$^*$ convergence.
\begin{definition}
Let $\alpha \in (0,1)$ be fixed.
We say that $\{w^n\}_n \subset \calF$ converges to $w \in \calF$ pseudo-weakly-$^*$ in $\calF$ 
and write $w^n \pweakstarto w$ if
\begin{gather*}
\begin{aligned}
\sum_{i=1}^M w_i \chi_{E_i^n} = w^n &\to w = \sum_{i=1}^M w_i \chi_{E_i} \text{ in } L^1(\Omega)\text{ and}\\
\sup_{n \in \N} P_\alpha(E_i^n) &< \infty \text{ for all } i \in \{1,\ldots,M\}.
\end{aligned}
\end{gather*}
We say that $\{w^n\}_n \subset \calF$ converges to $w \in \calF$ strictly in $\calF$ if
\begin{gather*}
\begin{aligned}
\sum_{i=1}^M w_i \chi_{E_i^n} = w^n &\to w = \sum_{i=1}^M w_i \chi_{E_i} \text{ in } L^1(\Omega)\text{ and}\\
R_\alpha(w^n) &\to R_\alpha(w).
\end{aligned}
\end{gather*}
\end{definition}
We obtain lower semicontinuity with respect to convergence in $L^1(\Omega)$ and in turn also for our regularizer,
which we briefly show below.
\begin{lemma}\label{lem:Ralpha_lsc}
Let $\alpha \in (0,1)$. Let $\chi_{E^n} \to \chi_E$ in $L^1(\Omega)$ for measurable sets 
$E^n$, $E \subset \Omega$. Then
\[ P_\alpha(E) \le \liminf_{n\to\infty} P_\alpha(E^n) . \]
In particular, for $w^n \to w$ in $L^1(\Omega)$ and $w^n$, $w \in \calF$ we obtain
\[ R_\alpha(w) \le \liminf_{n\to\infty} R_\alpha(w^n). \]
\end{lemma}
\begin{proof}
Clearly, the first claim holds if $\liminf_{n \to \infty} P_\alpha(E^n;\Omega) = \infty$. If this is not the
case, we observe that for $s = 0.5 \alpha$ and by means of $|\chi_E(x) - \chi_E(y)| = |\chi_E(x) - \chi_E(y)|^2$
it holds that
\[ P_\alpha(E) 
   \underset{\eqref{eq:frac_perim_as_one_double_integral}}=
   \int_{\R^d} \int_{\R^d}
   \frac{|\chi_E(x) - \chi_E(y)|^2}{|x - y|^{d + 2s}}
   \dd x \dd y,
\]
where we have extended $\chi_E$ with the value zero outside of $\Omega$ and where the right hand side is the
squared Gagliardo seminorm of the Hilbert space $W^{s,2}(\R^d)$, see, for example, \cite{di2012hitchhikers}.
Moreover, the boundedness of the characteristic functions gives $\chi_{E^n} \to \chi_E$ in $L^2(\Omega)$ too.
We infer that all subsequences $\chi_{E^{n_k}}$ such that $P_\alpha(E^{n_k})$ is bounded converge
weakly to $\chi_E$. Then the weak lower semicontinuity
of the seminorm of $W^{s,2}(\R^d)$  yields the first claim.
The second claim follows directly from the first and the definition of $R_\alpha$.
\end{proof}

\subsection{Local variations}
We follow the ideas presented in \cite{manns2023on} in order to 
derive stationarity conditions for \eqref{eq:p} and obtain
a corresponding sufficient decrease that in turn allows to prove
convergence of a sufficient decrease condition.
Both rely on the analysis of a perturbation of the partition
$E_1$, $\ldots$, $E_M$. We introduce such perturbations by
means of so-called \emph{local variations}, where we follow
\cite{manns2023on}, which in turn is based on
\cite{maggi2012sets}.
\begin{definition}[Definition 3.1 in \cite{manns2023on}]
\begin{enumerate}[label=\emph{(\alph*)}]
\item A \emph{one-parameter family of diffeomorphisms of $\R^d$}
is a smooth function $f:  (-\varepsilon,\varepsilon)  \times \R^d\to\R^d$
for some $\varepsilon > 0$ such that for all $t \in (-\varepsilon,\varepsilon)$,
the function $f_t(\cdot) \coloneqq f(t,\cdot) : \R^d \to\R^d$ is a diffeomorphism.
\item Let $A \subset \R^d$ be open. Then the family
$(f_t)_{t\in(-\varepsilon,\varepsilon)}$ is a
\emph{local variation in $A$} if in addition to (a) we have
$f_0(x) = x$ for all $x \in \R^d$ and there is a compact set $K \subset A$
such that $\{x \in \R^d\,|\,f_t(x) \neq x\} \subset K$
for all  $t \in (-\varepsilon,\varepsilon)$.
\item For a local variation, we define its \emph{initial velocity}:
$\phi(x) \coloneqq \frac{\partial f}{\partial t}(0, x)$ for $x \in \R^d$.
\end{enumerate}
\end{definition}
\begin{proposition}[Proposition 3.2 in \cite{manns2023on}]\label{prp:elementary_local_variation_properties}
Let $\phi \in C_c^\infty(\Omega,\R^d)$. Let
$f_t \coloneqq I + t\phi$ for $t \in \R$.
Then $(f_t)_{t\in(-\varepsilon,\varepsilon)}$
is a local variation in $\Omega$ with initial velocity $\phi$
for some $\varepsilon > 0$.
\end{proposition}
\begin{proof}
We refer the reader to Proposition 3.2 in \cite{manns2023on}.
\end{proof}
Let $w = \sum_{i=1}^M w_i \chi_{E_i}$ and
$(f_t)_{t \in (-\varepsilon,\varepsilon)}$ be a local
variation in $\Omega$. Then we define the functions
\[ f_t^{\#}w \coloneqq \sum_{i=1}^M w_i \chi_{f_t(E_i)} \]
for all $t \in (-\varepsilon,\varepsilon)$.
For the results below, we consider an arbitrary but
fixed local variation $(f_t)_{t \in (-\varepsilon,\varepsilon)}$
in $\Omega$ with initial velocity
$\phi \in C_c^\infty(\Omega,\R^d)$.
\begin{proposition}
Let $\{E_1,\ldots,E_M\}$ be a partition of $\Omega$.
Then the transformed sets $\{f_t(E_1),\ldots,f_t(E_M)\}$ are a
partition of $\Omega$ for all $t \in (-\varepsilon,\varepsilon)$.
\end{proposition}
\begin{proof}
This follows because the $f_t$ are diffeomorphisms,
see also \cite[\S3]{manns2023on}.
\end{proof}
We mention that local variations also allow for a notion of stability, see Definition 1.6 in \cite{cinti2019quantitative}.
\begin{proposition}[Variation of the fractional perimeter]\label{prp:variation_fractional_perimeter}
Let $\phi \in C_c^\infty(\Omega;\R^d)$ and let $E \subset \Omega$ satisfy $P_\alpha(E) < \infty$.
Then, 
\begin{gather*}
P_\alpha(f_t(E)) - P_\alpha(E) =
t L_\alpha(E,\phi)
   + o(t) ,
\end{gather*}
holds with the definitions
\[  
L_\alpha(E,\phi) \coloneqq 
    \int_E \int_{E^c}
    \frac{C(x,y)}{|x - y|^{d + \alpha}} \dd x \dd y  < \infty
\]
and
\[ C(x,y)
   \coloneqq
   \dvg \phi(x) 
   + \dvg \phi(y)
   - (d + \alpha)\frac{(x - y)\cdot (\phi(x) - \phi(y)}{|x - y|^2}. 
\]
\end{proposition}
\begin{proof}
By following the arguments in the proof of Lemma 4.1.1
with the choice $s = \alpha / 2$ on
p.\ 182 in \cite{dipierro2016nonlocal}, we obtain the
desired identity
\begin{gather*}
P_\alpha(f_t(E)) - P_\alpha(E) =
   t
    \int_E \int_{E^c}
    \frac{C(x,y)}{|x - y|^{d + \alpha}}
   \dd x \dd y
   + o(t).
\end{gather*} 
Before proving that $|C(x,y)|$ is uniformly bounded, which concludes the proof,
we briefly extend an argument in the proof of Lemma 4.1.1 in \cite{dipierro2016nonlocal}.
We first note that by means of Fubini's theorem and the substitution formula,
see for example Theorem 263D (v) in \cite{fremlinmeasure}, we obtain that $P_\alpha(f_t(E))$ is finite
if and only if
\begin{gather}\label{eq:tobefinite}
\int_E \int_{E^c} \frac{1}{|f_t(x) - f_t(y)|^{d + \alpha}}
|\det D f_t(x)| |\det D f_t(y)| \dd x \dd y
\end{gather}
is finite, where $D f_t = I + t \nabla \phi$.

We do this because we did not immediately see why the $o(\varepsilon)$-term ($o(t)$ in our
notation) in its proof that is due to the remainder term of the Taylor expansion
of $|x - y|^{-d - \alpha}$ is an $o(t)$-term even after integrating. Loosely speaking,
why were the authors of \cite{dipierro2016nonlocal} able to deduce
$\int_{E}\int_{E^c} o(t) \dd x \dd y = o(t)$? Since we did not
see the argument directly ourselves, we provide the arguments in more detail below for convenience.
They show both that the arguments in \cite{dipierro2016nonlocal} are correct and we just required
some more steps and, moreover, that the arguments from \cite{dipierro2016nonlocal}
even allow to prove the claim for the assumed generality $P_\alpha(E) < \infty$.

Applying Lemma 17.4 from \cite{maggi2012sets} and inspecting its proof (note that
$\phi$ has compact support and is Lipschitz), we obtain that
there exist $t_0 > 0$ and functions $c_1(t)$, $c_2(t)$ that are uniformly bounded and whose bounds
depend only on $t_0$ and $\phi$ such that for all $t$ with $|t| < t_0$ we obtain for all $x$, $y \in \R^d$
\begin{align}
|\det D f_t(x)| |\det D f_t(y)|
&= (1 + t \dvg \phi (x) + t^2 c_1(t))(1 + t \dvg \phi (y) + t^2 c_1(t)) \nonumber \\
&= (1 + t \dvg \phi (x))(1 + t \dvg \phi (y))+ t^2 c_2(t).\label{eq:detdet_bound}
\end{align}
Moreover, we deduce for $x \neq y$
\begin{align*}
\frac{1}{|f_t(x) - f_t(y)|^{d + \alpha}}
&= \frac{1}{|x - y + t(\phi(x) - \phi(y))|^{d + \alpha}} \\
&= \frac{1}{|x - y|^{d + \alpha}}
   - t \underbrace{(d + \alpha)\frac{(x - y)^T(\phi(x) - \phi(y))}{|x - y|^{d + \alpha + 2}}}_{\eqqcolon \ell(x,y)}
   + \tilde{r}(x,y),
\end{align*}
where the remainder $\tilde{r}(x,y)$ term of the Taylor expansion of $z \mapsto |z|^{-d-\alpha}$
at $x - y$ is given by
\[\tilde{r}(x,y) 
= t^2 (\phi(x) - \phi(y))^T
\left[
\frac{(d + \alpha + 2)(d + \alpha)}{|\xi|^{d + \alpha + 4}} \xi \xi^T 
-\frac{d + \alpha}{|\xi|^{d + \alpha + 2}} I 
\right] (\phi(x) - \phi(y))
\]
for some $\xi = \xi(x,y)$ in the line segment between $[x - y, x - y + t (\phi(x) - \phi(y))$.
Because $\phi$ is compactly supported and Lipschitz continuous with Lipschitz constant $L$,
there exists $0 < t_1 \le t_0$ so that for all $|t| \le t_1$ it holds that
$|\xi(x,y)|\ge \tfrac{1}{2} |x - y|$ uniformly for all $x$, $y \in \R^d$.
Combining this with the Lipschitz continuity of $\phi$, we deduce
\[
\tilde{r}(x,y) \le t^2 L^2 \left[
\frac{(d + \alpha + 2)(d + \alpha)}{2^{d + \alpha + 4}}
+ \frac{d + \alpha}{2^{d + \alpha + 2}}
\right]\frac{1}{|x - y|^{d + \alpha}}.
\]
Computing the double integral over the remainder term of the Taylor expansion implies
for $r(t) \coloneqq \int_{E}\int_{E^c} \tilde{r}(x,y) \dd x \dd y$ that
\[
|r(t)| = \left|\int_{E}\int_{E^c} \tilde{r}(x,y) \dd x \dd y\right|
\le t^2 c_3 P_\alpha(E) \]
for some $c_3 > 0$ and all $|t| \le t_1$, which gives that $r(t) = o(t)$.
With a similar argument, one can deduce that there exists $c_4 > 0$ such that
\[ \left|\int_{E}\int_{E^c} \ell(x,y) \dd x \dd y \right| \le c_4 P_\alpha(E).
\]
Combining these considerations with \eqref{eq:detdet_bound} and using
that $\phi$ is compactly supported therein,
we obtain that \eqref{eq:tobefinite} is finite and in turn also
$P_\alpha(f_t(E))$ is finite for all $t$ with $|t| \le t_1$.

Due to these arguments, we can combine all $o(t)$ terms and obtain the claimed formula for $P_\alpha(f_t(E)) - P_\alpha(E)$.
It remains to show that $|C(x,y)|$ is uniformly bounded on $E \times E^c$.
Because $\phi \in C_c^\infty(\Omega;\R^d)$, $|\dvg \phi(x)|$ and $|\dvg \phi(y)|$ are uniformly
bounded. Moreover, because $\phi \in C_c^\infty(\Omega;\R^d)$  is globally Lipschitz
with Lipschitz constant $L \ge 0$, we obtain that the absolute value of the third term
is bounded by $L (d + \alpha)$ in combination with the Cauchy--Schwarz inequality.
\end{proof}
We note that an alternative proof of this claim follows
from the arguments in Section 3 of \cite{maggi2017capillarity}. We also note that the
strategy above uses the same basic steps as the arguments in Section 17 of
\cite{maggi2012sets} for the limit case $\alpha = 1$.

\begin{lemma}\label{lem:sym_diff_boundedness}
Let $\phi \in C_c^\infty(\Omega, \R^d)$. Let $(f_t)_{t\in(-\varepsilon,\varepsilon)}$ be the local variation defined by $f_t \coloneqq I + t \phi$ for $t \in (-\varepsilon,\varepsilon)$.
Then there exist $0 < \varepsilon_0 < \varepsilon$ and $L > 0$ such that for all $s$, $t \in (-\varepsilon_0,\varepsilon_0)$, and measurable sets
$E$, $F \subset \Omega$ with $P_\alpha(E) < \infty$
it holds that
\[ \left|\lambda(f_t(E)\cap F) - \lambda(E \cap F)\right|
\le L |t|^\alpha P_\alpha(E) \]
and in particular,
\[ \lambda( f_t(E) \symdiff E) \le L |t|^\alpha P_\alpha(E). \]
\end{lemma}
\begin{proof}
We follow the proof strategy of Lemma 17.9 in \cite{maggi2012sets}. To this end,
let $g_t \coloneqq f_t^{-1}$
for all $t \in (-\varepsilon,\varepsilon)$ and let $s$, $t \in (-\varepsilon,\varepsilon)$.
Let $u_\delta \to \chi_E$ in $W^{\alpha,1}(\R^d)$ with
$u_\delta \in C_c^\infty(\R^d)$, which exists by Theorem A.1 in \cite{comi2019distributional}.
Then we obtain
\begin{align}
\left|\lambda(f_t(E)\cap F) - \lambda(E \cap F)\right|
&\le \int_{F} |\chi_E(g_{t}(x)) - \chi_E(x)|\dd x \label{eq:int} \\
&=\lim_{\delta \to 0} \underbrace{\int_{F} |u_\delta(g_{t}(x)) - u_\delta(x)|\dd x}_{\eqqcolon d_\delta},\label{eq:mollified_int}
\end{align}
where we note that the right hand side of the inequality \eqref{eq:int}
is identical to $\lambda( f_t(E) \Delta E)$ for the choice $F = \Omega$
so that the second claim follows from the first in this case.

Because the function $u_\delta$ is smooth, we can follow the proof of Proposition
3.14 in \cite{comi2019distributional} (which in turn relies on the fractional 
fundamental theorem of calculus as it is given in Theorem 3.12 in
\cite{comi2019distributional})
in order to obtain
\[ d_\delta \le
  \gamma_{d,\alpha} \sup_{x \in \R^d}
  \|\phi_t(x) - x\|^\alpha
  \|\nabla^\alpha u_\delta\|_{L^1(\R^d,\R^d)},
\]
where $y$ in the proof of Proposition 3.14 in
\cite{comi2019distributional} is replaced by $\phi_t(x) - x$,
H\"{o}lder's inequality is applied to obtain the
term $\sup_{x \in \R^d}
\|\phi_t(x) - x\|^\alpha$
and the constant $\gamma_{d,\alpha}$ is from Proposition 3.14
in \cite{comi2019distributional} too.
Then, the identity
\[ \sup_{x\in \R^d} \|\phi_t(x) - x\|
   =  |t| \sup_{x \in \R^d}\|\phi(x)\|
\]
and the fact that $\|\phi(x)\|$ is bounded because $\phi \in C^\infty_c(\Omega,\R^d)$ 
yield
\[ d_\delta \le
c \gamma_{d,\alpha} |t|^\alpha
\|\nabla^\alpha u_\delta\|_{L^1(\R^d,\R^d)}
\le c \gamma_{d,\alpha} \mu_{d,\alpha} |t|^\alpha
    \underbrace{\int_{\R^d}\int_{\R^d} \frac{|u_\delta(x) - u_\delta(y)|}{|x - y|^{d + \alpha}}\dd x\dd y}_{\eqqcolon [u_\delta]_{W^{\alpha,1}(\R^d)}}
\]
for some $c > 0$ and the constant $\mu_{d,\alpha}$ from (1.2) in 
\cite{comi2019distributional}.

Since $[u_\delta]_{W^{\alpha,1}(\R^d)}$ is the seminorm of $W^{\alpha,1}(\R^d)$
and $u_\delta$ approximates $\chi_E$ in norm in $W^{\alpha,1}(\R^d)$, we
obtain
\[ [u_\delta]_{W^{\alpha,1}(\R^d)}
   \to [\chi_E]_{W^{\alpha,1}(\R^d)} = P_\alpha(E).
\]
\end{proof}

\begin{proposition}[Variation of the linearized objective]\label{prp:variation_linearized_objective}
Let $g \in C^2(\bar{\Omega})$.
Let $E \subset \Omega$ satisfy $P_\alpha(E) < \infty$.
Let $\phi \in C_c^\infty(\Omega, \R^d)$. Let $(f_t)_{t\in(-\varepsilon,\varepsilon)}$ be the local variation defined by $f_t \coloneqq I + t \phi$ for $t \in (-\varepsilon,\varepsilon)$.
Then,
\[
(g, \chi_{f_t(E)} - \chi_{E})_{L^2}
= t \int_E \dvg (g(x) \phi(x)) \dd x
+ O(t^2).
\]
\end{proposition}
\begin{proof}
The (first half of the) arguments in the proof of Proposition 17.8
in \cite{maggi2012sets} imply
\[
(g, \chi_{f_t(E)} - \chi_{E})_{L^2}
= t \int_E \dvg (g(x) \phi(x)) \dd x
+ O(t^2),
\]
where $\int_E \dvg (g(x) \phi(x)) \dd x$ exists and is finite because of the assumed regularity of $g$.
\end{proof}
\begin{remark}
We note that the assumed regularity on $g$ for the Taylor expansion of the Lebesgue measure
in \cref{prp:elementary_local_variation_properties} may deemed to be unrealistically high.
Thus improving the required regularity is an important question for further research, in
particular, because the much weaker requirement $g \in C(\bar{\Omega})$ is sufficient for
proving a Taylor expansion in the non-local case, see \cite[Proposition 17.8]{maggi2012sets}.
We have, however, not found a means to do so until this point.
\end{remark}

\section{Standing assumptions and existence of solutions}\label{sec:existence}
We provide two standing assumptions on our problem that allow us to deduce the existence of solutions
as well as first-order optimality conditions for \eqref{eq:p} and \eqref{eq:p0} .
\begin{assumption}\label{ass:standing}~
\begin{enumerate}
\item[A.1] Let $F : L^1(\Omega) \to \R$ be bounded below.
\item[A.2] Let $F : L^2(\Omega) \to \R$ be twice continuously Fr\'{e}chet
differentiable. For some $C > 0$ and all $\xi \in L^2(\Omega)$, let the
bilinear form induced by the Hessian $\nabla^2 F(\xi) : L^2(\Omega) \times L^2(\Omega) \to \R$
satisfy $|\nabla^2 F(\xi)(u,w)| \le C\|u\|_{L^1}\|w\|_{L^1}$.
\end{enumerate}
\end{assumption}
The second part of \cref{ass:standing} is rather restrictive. The estimate on the Hessian
with respect to the $L^1$-norms is, for example, not satisfied if $F(w) = \tfrac 12 \|w\|_{L^2}^2$
and it implies that $F$ involves some Lipschitz operation that maps $L^1(\Omega)$
to a smaller space. This may be a convolution operator or a solution operator of a PDE, see
also the discussions in \cite{leyffer2022sequential,manns2023on,manns2023convergence,manns2023homotopy}.

The existence of minimizers for the limit problem \eqref{eq:p0} 
under \cref{ass:standing} follows as in \cite{leyffer2022sequential,manns2023on}.
The existence of minimizers for the problems \eqref{eq:p} follows from the compactness
and lower semicontinuity properties of nonlocal perimeters, see, for example,
\cite[Section 3.7]{cozzi2017regularity}. We briefly recap how the existence
is achieved below.
\begin{proposition}\label{prp:existence}
Let $\alpha \in (0,1)$.
Let \cref{ass:standing} hold. Then \eqref{eq:p} admits a minimizer
$w = \sum_{i=1}^M w_i \chi_{E_i}$
\end{proposition}
\begin{proof}
We apply the direct method of calculus of variations. There is a real infimum
of \eqref{eq:p} because all terms of the objective are bounded. We thus consider a minimizing sequence
$w^n = \sum_{i=1}^M w_i \chi_{E_i^n}$ with a corresponding
sequence of partitions $(\{E_1^n,\ldots,E_M^n\})_n$ of $\Omega$. Because all terms
of the objective are bounded below, we obtain that the sequences 
$(P_\alpha(E_i^n))_n$ are bounded. Consequently, the sequences $(\chi_{E_i^n})_n$ are bounded
in a space that, similar to $\BV(\Omega)$, admits sequential weak-$^*$ compactness properties, see
\cite[Corollary 4.6 and Proposition 4.8]{comi2019distributional}, which
yields a limit function $w = \sum_{i=1}^M w_i \chi_{E_i}$ such that
$\chi_{E^n_i} \to \chi_{E_i}$ in $L^1(\Omega)$ for all $i \in \{1,\ldots,M\}$.
Because convergence in $L^1(\Omega)$ implies pointwise a.e.\ convergence 
for a subsequence, we obtain that the limit sets $E_i$ are a partition
of $\Omega$ except for a set of Lebesgue measure zero. In other words, $w \in \calF$.
Moreover, Lebesgue's dominated convergence theorem gives $w^n \to w$ in $L^2(\Omega)$.

Because $R_\alpha$ is lower semicontinuous with respect to convergence
in $L^1(\Omega)$ on our feasible set, see \cref{lem:Ralpha_lsc}, and $F$
is continuous, see \cref{ass:standing}, we obtain that $w$ minimizes \eqref{eq:p}.
\end{proof}

\section{Compactness and $\Gamma$-convergence for $\alpha \nearrow 1$}\label{sec:gamma}
We deduce compactness in $L^1(\Omega)$ and a $\Gamma$-convergence-type result from the results
in \cite{ambrosio2010gamma}. It is immediate that $P_\alpha(E)$ as
defined in \eqref{eq:fractional_perimeter} equals $\mathcal{J}^1_s(E, \R^d)$
in \cite{ambrosio2010gamma} and $\mathcal{J}^2_s(E, \R^d)$ in \cite{ambrosio2010gamma} 
equals zero for all $E \subset \Omega$. We start with the compactness result, continue with the
$\liminf$- and $\limsup$-inequalities for $\Gamma$-convergence, and conclude that global minimizers converge
to global minimizers.
\begin{proposition}[Corollary of Theorem 1 in \cite{ambrosio2010gamma}]\label{prp:compactness_ato1}
Let \cref{ass:standing} hold. Let $\alpha \nearrow 1$. Let $w^\alpha$ be feasible for \eqref{eq:p}.
Let $\sup_{\alpha \nearrow 1} J_\alpha(w^\alpha) < \infty$.
Then $\{w^\alpha\}_\alpha$ is relatively compact in $L^1(\Omega)$,
admits a limit point, and all limit points are $W$-valued.
\end{proposition}
\begin{proof}
Because $F$ and the summands in $R_\alpha$ are bounded below
it follows that $\sup_{\alpha \nearrow 1} (1 - \alpha)P_\alpha(E_i^\alpha) < \infty$
for all $i \in \{1,\ldots,M\}$, where $E_i^\alpha = (w^\alpha)^{-1}(\{w_i\})$.
Consequently, we can apply Theorem 1 in \cite{ambrosio2010gamma} in order to obtain that
for all $i \in \{1,\ldots,M\}$, the sequence $(\chi_{E_i}^\alpha)_\alpha$ is relatively
compact in $L^1_{loc}(\R^d)$ and in turn in $L^1(\Omega)$.
Recalling $w^\alpha = \sum_{i=1}^M w_i \chi_{E_i^\alpha}$, we obtain
that the sequence $(w^\alpha)_\alpha$ is relatively compact.
Because convergence in $L^1(\Omega)$ implies pointwise convergence
a.e.\ for a subsequence, we obtain that the limit is
$W$-valued, see also \cite{leyffer2022sequential}.
\end{proof}

\begin{proposition}[Corollary of Theorem 2 in \cite{ambrosio2010gamma}]\label{prp:liminf_limsup_ato1}
Let $w = \sum_{i=1}^M w_i \chi_{E_i}$ for a partition $\{E_1,\ldots,E_M\}$
of $\Omega$ into measurable sets $E_i$.
Then:
\begin{enumerate}
\item if $w^\alpha = \sum_{i=1}^M w_i \chi_{E^\alpha_i} \to w$ in $L^1(\Omega)$
for $i \in \{1,\ldots,M\}$ for partitions $\{E_1^\alpha,\ldots,E_M^\alpha\}$ of $\Omega$,
it holds that $J(w) \le \liminf_{\alpha \nearrow 1} J_\alpha(w)$,
\item there is a sequence $w^\alpha = \sum_{i=1}^M w_i \chi_{E^\alpha_i} \to w$ in $L^1(\Omega)$
for $i \in \{1,\ldots,M\}$ for partitions $\{E_1^\alpha,\ldots,E_M^\alpha\}$ of $\Omega$
such that $J(w) \ge \limsup_{\alpha \nearrow 1} J_\alpha(w^\alpha)$.
\end{enumerate}
\end{proposition}
\begin{proof}
The first claim ($\liminf$-inequality) follows directly from Theorem 2 in \cite{ambrosio2010gamma}.
The second claim ($\limsup$-inequality) is immediate for $R(w) = \infty$. If $R(w) <\infty$, the
partition $\{E_1,\ldots,E_M\}$ is a so-called Caccioppoli partition. 

Then $\{E_1,\ldots,E_M\}$ is a Caccioppoli partition of $\Omega$ and \cite{braides2017density}
asserts that polyhedral partitions are dense in the Caccioppoli partitions, that is there exist
sets $\{T_1^k,\ldots,T_M^k\}_k \subset \R^d$ whose boundaries are composed of finitely many
convex polytopes such that 
\begin{align}
\sum_{i=1}^M w_i \chi_{T_i^k} &\to \sum_{i=1}^M w_i \chi_{E_i} \text{ in } L^1(\Omega) \text{ for } k \to \infty,
\label{eq:l1_approximation}\\
P(T_i^k \cap \Omega;\R^d) &\to P(E_i;\R^d) \text{ for } k \to \infty \text{ and all } i \in \{1,\ldots,M\},
\label{eq:perimeter_approximation}
\end{align}
where the second convergence follows from Corollary 2.5 in 
\cite{braides2017density} in combination with
$P(T_i^k \cap \Omega;\R^d) = \Ha^{d-1}(\Omega \cap \partial^* T_i^k)
+ \Ha^{d-1}(\partial \Omega \cap \overline{T_i^k})$, where
$\partial^* A$ of a set $A \subset \R^d$ denotes its reduced boundary.

Because of \eqref{eq:l1_approximation} and \eqref{eq:perimeter_approximation}, 
we can assume that all $T_i^k$ are contained in a bounded
\emph{hold-all domain}, e.g., a large ball.
Because the sets $T_i^k \cap \Omega$ are polyhedral (recall that we assumed
that $\Omega$ is polyhedral at the beginning of the article),
we can apply Lemmas 8 and 9 in  \cite{ambrosio2010gamma} to them and obtain
\begin{gather}\label{eq:polyhedral_perimeter_recovery}
\lim_{\alpha \nearrow 1} (1 - \alpha) P_\alpha(T_i^k \cap \Omega) = \omega_{n - 1}P(T_i^k \cap \Omega).
\end{gather}
Combining \eqref{eq:perimeter_approximation} and \eqref{eq:polyhedral_perimeter_recovery}, we choose a suitable diagonal sequence
indexed by $(k_n)_n$ of the polyhedral partitions $\{T_1^k \cap \Omega,\ldots,T_M^k \cap \Omega\}_k$ in order to obtain
\[ \lim_{\alpha \nearrow 1} P_\alpha(T_i^{n_k} \cap \Omega; \Omega) = \omega_{n-1} P(E_i;\Omega),
\]
which implies the assertion.
\end{proof}

We note that the proof of \cref{prp:liminf_limsup_ato1} is the only point
in this article, where we use the assumption that the domain $\Omega$ is
polyhedral. As a corollary of the compactness established in 
\cref{prp:compactness_ato1} and the liminf- and limsup-inequalities
established in \cref{prp:liminf_limsup_ato1}, we obtain
the convergence of global minimizers to global minimizers below.

\begin{corollary}
Let \cref{ass:standing} hold. Let $\alpha \nearrow 1$. Let $w^\alpha$ be a global minimizer
of \eqref{eq:p} for $\alpha$. Then there exists a $W$-valued accumulation point $w \in L^1(\Omega)$
with $\tilde{w} \in \BV(\R^d)$, where $\tilde{w}$ is the extension of $w$ by zero outside
of $\Omega$, $w(x) \in W$ a.e., such that $w^\alpha \to w$ in $L^1(\Omega)$ and $J_\alpha(w^\alpha) \to J(w)$.
Moreover, for all accumulation points of $(w^\alpha)_{\alpha}$ are $W$-valued,
global minimizers of \eqref{eq:p0}, and their extensions by zero
$(\tilde{w}^\alpha)_{\alpha}$ outside of $\Omega$ are in $\BV(\R^d)$.
\end{corollary}

\section{Local minimizers and trust-region algorithm}\label{sec:stat_and_tr}
As is noted in \cite{leyffer2022sequential,manns2023on}, it makes sense to consider
local minimizers and stationary points in the settings of \eqref{eq:p} and
\eqref{eq:p0} because $L^1$-neighborhoods of feasible points contain further
feasible points. We briefly translate these concepts from the perimeter-regularized
case, see \cite{manns2023on}, to our setting in \cref{sec:loc_min}.
Then we introduce and analyze non-local variants of the trust-region subproblem
from \cite{manns2023on} in \cref{sec:tr_sub}. We introduce and describe a
trust-region algorithm that builds on these subproblems in \cref{sec:tr_alg}
and prove its asymptotics in \cref{sec:tr_alg_asymptotics}.

\subsection{Local minimizers and stationary points}\label{sec:loc_min}
We start by defining local minimizers and stationary points
and then verify that local minimizers are stationary.
\begin{definition}
Let $\alpha \in (0,1)$ and let $w = \sum_{i=1}^M w_i \chi_{E_i}$ be feasible for \eqref{eq:p}.
\begin{itemize}
\item We say that $w$ is \emph{locally optimal} for \eqref{eq:p} if there is $r > 0$
such that $J_\alpha(w) \le J_\alpha(v)$ for all $v$ that are feasible for \eqref{eq:p}
and satisfy $\|v - w\|_{L^1} \le r$.
\item Let $\alpha \in (0.5,1)$. Let $\nabla F(w) \in C^2(\bar{\Omega})$. We say that $w$
is \emph{stationary} for \eqref{eq:p} if
\begin{gather}\label{eq:stationarity}
\sum_{i=1}^M w_i \left(\int_{E_i} \dvg (\nabla F(w) (x) \phi(x)) \dd x 
+ L_\alpha(E_i,\phi)\right)
= 0
\end{gather}
for all $\phi \in C_c^\infty(\Omega;\R^d)$.
\end{itemize}
\end{definition}
\begin{proposition}\label{prp:minimizers_are_stationary}
Let \cref{ass:standing} hold.
Let $\alpha \in (0.5,1)$. $w = \sum_{i=1}^M w_i \chi_{E_i}$ be locally 
optimal for \eqref{eq:p}. Then $w$ is stationary for \eqref{eq:p}.
\end{proposition}
\begin{proof}
Let $\phi \in C_c^\infty(\Omega,\R^d)$. Let $(f_t)_{t \in (-\varepsilon,\varepsilon)}$
for some $\varepsilon > 0$ be the local variation defined by $f_t \coloneqq I + t \phi$.
the function $t \mapsto F(f_t^{\#}w) + \eta R_\alpha( f_t^{\#}w )$
is differentiable at $t = 0$, which can be seen as follows.
\Cref{ass:standing} implies that
\[ \frac{\dd}{\dd t}F(f_t^{\#}w)\Big|_{t = 0}
   = \frac{\dd}{\dd t} (\nabla F(w), f_{t}^{\#}w)_{L^2}\Big|_{t = 0}
   + \frac{\dd}{\dd t} \underbrace{\nabla^2F(\xi^t)(f_t^{\#}w - w, f_t^{\#}w - w)}_{\eqqcolon r_t}
   \Big|_{t = 0}
\]
for some $\xi^t$ in the line segment between $w$ and $f_t^{\#}w$.
\Cref{ass:standing} and \cref{lem:sym_diff_boundedness} imply
\[ \left|
\nabla^2F(\xi^t)(f_t^{\#}w - w, f_t^{\#}w - w)
\right| \le C {|t|}^{2\alpha} \sum_{i=1}^M P_\alpha(E_i;\Omega)
\]
for some large enough $C > 0$ and all $t \in (-\varepsilon,\varepsilon)$.
Because $\alpha > 0.5$, we obtain that $r_t$ is differentiable at $t = 0$
with value zero so that
\[ \frac{\dd}{\dd t}F(f_t^{\#}w)\Big|_{t = 0}
   = \frac{\dd}{\dd t} (\nabla F(w), f_{t}^{\#}w)_{L^2}\Big|_{t = 0}. \]
Then \eqref{eq:stationarity} follows from  
\cref{prp:variation_linearized_objective,prp:variation_fractional_perimeter}.   
\end{proof}

\subsection{Trust-region subproblems}\label{sec:tr_sub}
We introduce trust-region and analyze subproblems by following the ideas from
\cite{leyffer2022sequential,manns2023on}, that is the principal
part of the objective enters the trust-region subproblem by means of a
linear model and the regularization term is considered exactly.
We analyze $\Gamma$-convergence of the trust-region subproblems
with respect to convergence of the linearization point, and provide
optimality conditions for the trust-region subproblem.

Let $\Delta > 0$ and $\bar{w}$ be feasible for \eqref{eq:p}
with $R_\alpha(\bar{w}) < \infty$. The trust-region subproblem reads
\begin{gather}\label{eq:tr}
\text{{\ref{eq:tr}}}(\bar{w}, g, \Delta) \coloneqq
\left\{
\begin{aligned}
\min_{w \in L^1(\Omega)}\ & (g, w - \bar{w})_{L^2} + \eta R_\alpha(w)- \eta R_\alpha (\bar{w})\\
\text{s.t.}\quad & \|w - \bar{w}\|_{L^1} \le \Delta,\\
                 & w(x) \in W \text{ for a.e.\ } x \in \Omega,
\end{aligned}
\right.%\}
\tag{TR$_\alpha$}
\end{gather}
where we recover the linearized principal part of the objective of
\eqref{eq:p} with the choice $g = \nabla F(\bar{w})$.
The trust-region subproblem $\text{{\ref{eq:tr}}}(\bar{w},g,\Delta)$ admits a
minimizer, which we briefly show below.
\begin{proposition}\label{prp:tr_well_defined}
	Let $\bar{w}$ be feasible for \eqref{eq:p} with $R_\alpha(\bar{w}) < \infty$,
	$g \in L^2(\Omega)$, and $\Delta \ge 0$.
	Then $\text{\emph{\ref{eq:tr}}}(\bar{w},g,\Delta)$ admits a minimizer.
\end{proposition}
\begin{proof}
Because $g \in L^2(\Omega)$ and $w$, $\bar{w} \in L^\infty(\Omega)$,
the first term of the objective of  $\text{{\ref{eq:tr}}}(\bar{w},g,\Delta)$
is bounded below.
Because of the $L^\infty(\Omega)$-bounds ($W$ is a finite set), convergence
in $L^1(\Omega)$ of feasible points implies convergence in $L^2(\Omega)$
and we obtain continuity of the first term of the objective if a sequence
of feasible points converges in $L^1(\Omega)$. Moreover, the term
$\eta R_\alpha (\bar{w})$ is constant. Consequently, the assumptions
of \cref{prp:existence} are satisfied on the non-empty feasible set
of $\text{{\ref{eq:tr}}}(\bar{w},g,\Delta)$. Thus the existence
of solutions to \eqref{eq:p} follows as a corollary from (the proof of)
\cref{prp:existence}.
\end{proof}
Moreover, if the linearization point minimizes $\text{{\ref{eq:tr}}}(w,\nabla F(w),\Delta)$,
it is stationary for \eqref{eq:p} too.
\begin{proposition}\label{prp:tr_stationary}
Let $\alpha \in (0.5,1)$. Let $\{E_1,\ldots,E_M\}$ be a partition of $\Omega$
such that $w = \sum_{i=1}^M w_i \chi_{E_i}$ satisfies $R_\alpha(w) < \infty$.
Let $\nabla F(w) \in C^2(\bar{\Omega})$. If $w$ is locally optimal
for $\text{\emph{\ref{eq:tr}}}(w,\nabla F(w),\Delta)$ for some $\Delta > 0$, then $w$ is
stationary for \eqref{eq:p}.
\end{proposition}
\begin{proof}
Let $g \coloneqq \nabla F(w)$. We choose $\tilde{F}(v) \coloneqq (g, v)_{L^2}$ for $v \in L^2(\Omega)$ and
obtain $\nabla \tilde{F}(v) = g$ and $\nabla^2 \tilde{F}(v) = 0$ so that $\tilde{F}$ satisfies \cref{ass:standing}
on $\calF$. We apply \cref{prp:minimizers_are_stationary} with $\tilde{F}$ for $F$ and obtain that
$w$ is stationary and satisfies \eqref{eq:stationarity} with $\nabla \tilde{F}(w) = g = \nabla F(w)$, which
means that $w$ is also stationary for \eqref{eq:p}.
\end{proof}
Next, we analyze $\Gamma$-convergence of the trust-region subproblems with respect to
strict and pseudo-weak-$^*$ convergence of feasible points of \eqref{eq:p}. As in
\cite{manns2023on}, this will be a key ingredient of our convergence analysis of our
trust-region algorithm.
\begin{theorem}\label{thm:tr_gamma_convergence}
	Let $\alpha \in (0,1)$.
	Let $v^n \to v$ strictly in $\calF$. Let $g^n \weakto g$ in $L^2(\Omega)$.
	Let $\Delta > 0$. Then the functionals $T^n : (\calF, \text{\emph{pseudo-weak}-$^*$}) \to \R$, defined as
	\[ 
	T^n(w) \coloneqq (g^n, w - v^n)_{L^2} + \eta R_\alpha(w) - R_\alpha(v^n) +
	\delta_{[0,\Delta]}(\|w - v^n\|_{L^1})
	\]
	for $w \in \calF$, $\Gamma$-converge to $T :  (\calF, \text{\emph{pseudo-weak}-$^*$}) \to \R$, defined as
	\[
	T(w) \coloneqq (g, w - v)_{L^2} + \eta R_\alpha(w) - R_\alpha(v)  +
	\delta_{[0,\Delta]}(\|w - v\|_{L^1})
	\]
	for $w \in \calF$, where $\delta_{[0,\Delta]}$ is the $\{0,\infty\}$-valued indicator function of $[0,\Delta]$.
\end{theorem}
\begin{proof}
We follow the proof strategy of \cite[Theorem 5.2]{manns2023on}.

\textbf{Part 1: Lower bound inequality. $T(w) \le \liminf_{n \to \infty} T^n(w^n)$
for $w^n \pweakstarto w$ in $\calF$.}\quad Because of the uniform $L^\infty(\Omega)$-bounds on $\calF$,
we obtain $w^n \to w$ in $L^2(\Omega)$ and $v^n \to v$ in $L^2(\Omega)$. In combination with
$g^n \weakto g$ in $L^2(\Omega)$, we obtain $(g^n, w^n - v^n)_{L^2} \to (g, w - v)_{L^2}$.

Because $w^n \pweakstarto w$ in $\calF$ and $v^n \to v$ strictly in $\calF$, we obtain
with the help of \cref{lem:Ralpha_lsc} that
$R_\alpha(w) - R_\alpha(v) \le \liminf_{n \to \infty} R_\alpha(w^n) - R_\alpha(v^n)$.
Moreover, if $\|w^{n_k} - v^{n_k}\|_{L^1} \le \Delta$ holds for an infinite subsequence,
then the convergence of $\{w^n\}_n$ and $\{v^n\}_n$ in $L^1(\Omega)$ and the triangle inequality
imply $\|w - v\|_{L^1} \le \Delta$ so that the last term of $T$ is zero and the lower bound
inequality is satisfied. If there is no such subsequence, then $T^n \equiv \infty$ and the
lower bound inequality holds trivially.

\textbf{Part 2: Upper bound inequality.} For each $w \in \calF$ with $R_\alpha(w)  < \infty$,
	there exists a sequence $w^n \pweakstarto w$ in $\calF$ such that
	$T(w) \ge \limsup_{n \to \infty} T^n(w^n)$.\quad We make a case distinction on the
	possible values of the norm difference $\|w - v\|_{L^1}$.

\textbf{Case 2a $\|w - v\|_{L^1} > \Delta$:}\quad
Then $T(w) = \infty$ and we can choose $w^n \coloneqq w$ for all $n \in \N$.

\textbf{Case 2b $\|w - v\|_{L^1} < \Delta$:}\quad
We choose again $w^n \coloneqq w$ for all $n \in \N$. We obtain
\begin{multline*}
(g^n, w^n - v^n)_{L^2} + \eta R_\alpha(w^n) - \eta R_\alpha(v^n) \\
\to (g, w - v)_{L^2} + R_\alpha(w) - R_\alpha(v) + \underbrace{\delta_{[0,\Delta]}(\|w - v\|_{L^1})}_{=0}.
\end{multline*}
The convergence of $\{w^n\}_n$ and $\{v^n\}_n$ in $L^1(\Omega)$ and the triangle inequality
imply that $\|w^n - v^n\|_{L^1} \le \Delta$ holds for all large enough $n \in \N$. Consequently,
$T^n(w^n) \to T(w)$.
	
\textbf{Case 2c $\|w - v\|_{L^1} = \Delta$:}\quad Because $\Delta > 0$, there exists a set
\[ D \coloneqq \{ x \in \Omega\,|\, v(x) = w_1 \neq w_2 = w(x) \} \]
with $\lambda(D) > 0$. We note that the specific values $w_1$ and $w_2$ are without loss of generality
because we may reorder the indices of the elements of $W$ as necessary. The set $D$ satisfies
\begin{align*}
P_\alpha(D) 
&= P_\alpha(v^{-1}(\{w_1\}) \cap w^{-1}(\{w_2\})) \\
&\le P_\alpha(v^{-1}(\{w_1\})) + P_\alpha(w^{-1}(\{w_2\}))\\
&\le R_\alpha(v) + R_\alpha(w) < \infty,
\end{align*}
where the first inequality follows from  \eqref{eq:union_intersection_inequality}.
The second inequality follows from the fact that at most one of the $w_i$ can be zero.
Because $D$ has strictly positive Lebesgue measure $\lambda(D) > 0$, it has a point of density $1$,
that is there exists $\bar{x} \in D$ such that
\begin{gather}\label{eq:pt_of_density_one}
\lim_{r \searrow 0} \frac{\lambda(D \cap B_r(\bar{x}))}{\lambda(B_r(\bar{x}))} = 1.
\end{gather}
We define $\kappa^n \coloneqq \|v^n - v\|_{L^1}$ for $n \in \N$. Because of \eqref{eq:pt_of_density_one}, there
exist a sequence $(r^n)_n$ and $n_0 \in \N$ such that
$r^n \searrow 0$ and for all $n \ge n_0$:
\[ \lambda(D \cap B_{r^n}(\bar{x})) \ge \kappa^n \ge \frac{1}{2}\lambda(B_{r^n}(\bar{x}))
\]
and $B_{r^n}(\bar{x}) \subset \Omega$. We now restrict to $n \ge n_0$ and
define $w^n$ by
\[ w^n(x) \coloneqq \left\{ 
\begin{aligned}
v(x) &\text{ if } x \in B_{r^n}(\bar{x}),\\
w(x) &\text{ else}
\end{aligned}
\right.
\]
for a.e.\ $x \in \Omega$. This gives
\begin{align*}
 \|w^n - v^n\|_{L^1} 
 &\le \|v - v^n\|_{L^1} + \|w^n - v\|_{L^1} \\
 &\le \kappa^n + \|w - v\|_{L^1(\Omega\setminus B_{r^n}(\bar{x}))} \\
 &\le \kappa^n + \Delta - \underbrace{|w_1 - w_2|}_{\ge 1}\underbrace{\lambda(D \cap B_{r^n}(\bar{x}))}_{\ge \kappa^n} 
      \underbrace{- \|w - v\|_{L^1({B_{r^n}(\bar{x}) \setminus D})}}_{\le 0} \le \Delta.
\end{align*}
The construction of the $w^n$ implies $w^n \pweakstarto w$ in $\calF$. In order to
obtain the $\limsup$-inequality, we show $R_\alpha(w^n) \to R_\alpha(w)$.
To this end, let $E_i \coloneqq w^{-1}(\{w_i\})$, $E_i^n \coloneqq (w^n)^{-1}(\{w_i\})$
for $i \in \{1,\ldots,M\}$ and $n \ge n_0$.
Then $E_1^n = E_1 \cup B_{r^n}(\bar{x})$ and
$E_i^n = E_i \setminus B_{r^n}(\bar{x})$ for $i \geq 2$.
For $\{E_1^n\}_n$, we deduce by means of \eqref{eq:union_intersection_inequality}
\begin{align*}
P_\alpha(E_1^n) &\le P_\alpha(E_1) + P_\alpha(B_{r^n}(\bar{x}))\\
&\le P_\alpha(E_1) + (r^{n})^{d - \alpha}P_\alpha(B_{1}(\bar{x}))
\to P_\alpha(E_1)
\end{align*}
and, analogously, by means of \eqref{eq:set_subtraction_inequality}
\begin{align*}
P_\alpha(E_i^n) \le P_\alpha(E_i) + (r^{n})^{d - \alpha}P_\alpha(B_{1}(\bar{x}))
\to P_\alpha(E_i)
\end{align*}
for all $i \in \{1,\ldots,M\}$. Summing the terms, we obtain 
$R_\alpha(w^n) \to R_\alpha(w)$.
\end{proof}

\subsection{Trust-region algorithm}\label{sec:tr_alg}
We propose to solve \eqref{eq:p} for locally optimal or stationary points with a variant of the
trust-region algorithm that is proposed and analyzed in \cite{leyffer2022sequential,manns2023on}.
It is stated as \cref{alg:trm} below and consists of two loops. The outer loop is indexed by
$n$ and in each iteration of the outer loop a new feasible iterate $w^n \in \calF$ with
$R_\alpha(w^n) < \infty$ is computed that improves acceptably over the previous iterate $w^{n-1}$.
An acceptable improvement is achieved if the new iterate $w^n$ satisfies
\begin{gather}\label{eq:accept}
\ared(w^{n-1},w^n) \ge \sigma \pred(w^{n-1},\Delta^{n,k})
\end{gather}
for a fixed $\sigma \in (0,1)$ and the trust-region radius $\Delta^{n,k}$ that is determined by the inner
loop (see below). In \eqref{eq:accept}, the left hand side is defined by
\[ \ared(w^{n-1}, w) \coloneqq F(w^{n-1}) + \eta R_\alpha(w^{n-1}) - F(w) - \eta R_\alpha(w) \]
for $w \in \calF$ and is the \emph{actual reduction} of the objective that is achieved by
$w$. The right hand side is the \emph{predicted reduction} that is achieved by the solution
$\tilde{w}^{n,k}$ of the trust-region subproblem $\text{{\ref{eq:tr}}}(w^{n-1}, \nabla F(w^{n-1}), \Delta^{n,k})$
\[ \pred(w^{n-1},\Delta^{n,k}) \coloneqq (\nabla F(w^{n-1}),w^{n - 1} - \tilde{w}^{n,k})
   + \eta R_\alpha(w^{n-1}) - \eta R_\alpha(\tilde{w}^{n,k})
\]
is the predicted reduction by the (negative objective of the) trust-region 
subproblem for the current trust-region radius and thus its solution 
$\tilde{v}^{n,k}$.

To this end, the inner loop, indexed by $k$, starts from the reset trust-region radius $\Delta^{n,0} = \Delta^0$ and
solves the trust-region subproblems $\text{{\ref{eq:tr}}}(w^{n-1}, \nabla F(w^{n-1}), \Delta^{n,k})$ with
linearization (model) point $w^{n-1}$. If the solution of the trust-region subproblem $\tilde{w}^{n,k}$
satisfies \eqref{eq:accept}, the new iterate $w^n$ is set as $\tilde{w}^{n,k}$ and the inner loop terminates.
Else, the trust-region radius is halved and the next iteration of the inner loop begins.
If the $\pred(w^{n-1},\Delta^{n,k}) = 0$, $w^{n-1}$ is a minimizer of the trust-region subproblem for
a positive trust-region radius and thus stationary for \eqref{eq:p} by virtue of \cref{prp:tr_stationary}
if $\alpha \in (0.5,1)$. In this case, \cref{alg:trm} terminates.

\begin{algorithm}[t]
	\caption{Trust-region Algorithm leaning on SLIP from \cite{leyffer2022sequential,manns2023on}}\label{alg:trm}
	\textbf{Input:}
	$\alpha \in (0,1)$,
	$F$ sufficiently regular,
	$\Delta^0 > 0$, $w^0 \in \calF$ with $R_\alpha(w^n) < \infty$, $\sigma \in (0,1)$.
	
	\begin{algorithmic}[1]
		\For{$n = 0,\ldots$}
		\State $k \gets 0$
		\State $\Delta^{n,0} \gets \Delta^0$
		\While{not sufficient decrease according to \eqref{eq:accept}}\label{ln:suffdec}
		\State\label{ln:trstep} $\tilde{w}^{n,k} \gets$ minimizer of $\text{{\ref{eq:tr}}}(w^{n-1}, \nabla F(w^{n-1}), \Delta^{n,k})$.
		\State\label{ln:pred} $\pred(w^{n-1},\Delta^{n,k}) \gets (\nabla F(w^{n-1}), w^{n-1} - \tilde{w}^{n,k})_{L^2}
		+ \eta R_\alpha(w^{n-1}) - \eta R_\alpha(\tilde{w}^{n,k})$
		\State $\ared(v^{n-1},\tilde{v}^{n,k}) \gets F(v^{n-1}) + \eta R_\alpha(w^{n-1})
		- F(\tilde{v}^{n,k}) - \eta R_\alpha(\tilde{w}^{n,k})$
		\If{$\pred(w^{n-1},\Delta^{n,k}) \le 0$}
		\State Terminate. The predicted reduction for $w^{n-1}$ is zero.
		\ElsIf{not sufficient decrease according to \eqref{eq:accept}}
		\State $k \gets k + 1$		
		\State $\Delta^{n,k} \gets \Delta^{n,k-1} / 2$.
		\Else
		\State $w^n \gets \tilde{w}^{n,k}$
		\EndIf
		\EndWhile	
		\EndFor
	\end{algorithmic}
\end{algorithm}

\subsection{Asymptotics of the trust-region algorithm}\label{sec:tr_alg_asymptotics}

With the results that have been established in the previous sections, the asymptotics of
\cref{alg:trm} can be analyzed by following the strategy from \cite{manns2023on},
which in turn is an extension of the analysis and ideas in \cite{leyffer2022sequential}.
We therefore only provide the information, where the proofs of the corresponding results
in \cite{manns2023on} require modification to match the situation of this work.
We begin with the proof of the asymptotics of the inner loop and continue
with the asymptotics of the outer loop.
\begin{proposition}[Corollary 6.3 in \cite{manns2023on}, Corollary 4.20 in \cite{leyffer2022sequential}]\label{prp:inner}
	Let $\alpha \in (0.5,1)$.
	Let \cref{ass:standing} hold. Let $w^{n-1}$ produced by \cref{alg:trm} 
	satisfy $\nabla F(w^{n-1}) \in C^2(\bar{\Omega})$. Then iteration
	$n$ satisfies one of the following outcomes.
	\begin{enumerate}
		\item The inner loop terminates after finitely many iterations and 
		\begin{enumerate}
			\item the sufficient decrease condition \eqref{eq:accept} is satisfied or
			\item the predicted reduction is zero (and the iterate $w^{n-1}$ is stationary
			for \eqref{eq:p}).
		\end{enumerate}
		\item The inner loop does not terminate and the iterate $w^{n-1}$ is stationary.
	\end{enumerate}
\end{proposition}
\begin{proof}
The proof follows as in  Corollary 6.3 with the major steps of the proof being Lemma 6.1 and Lemma 6.2 in \cite{manns2023on},
where the violation of \emph{L-stationarity} in Lemma 6.2
 is replaced by a violation of \eqref{eq:stationarity} and the $\TV$-term is
 replaced by $R_\alpha$.
The roles of Lemma 3.3, Lemma 3.5, and Proposition 5.5 in \cite{manns2023on} are taken by
\cref{prp:variation_fractional_perimeter,prp:variation_linearized_objective,prp:minimizers_are_stationary}.
The role of Lemma 3.8 in \cite{manns2023on} is taken by \cref{lem:sym_diff_boundedness}.
It leads to the term $(\varepsilon^k)^{2\alpha}$ instead of $(\varepsilon^k)^2$ in the proof of Lemma 6.2,
which is still dominated by $\varepsilon^k \eta$ for $\varepsilon^k \searrow 0$ if $\alpha \in (0.5,1)$
as assumed.
\end{proof}

\begin{theorem}[Theorem 6.4 in \cite{manns2023on}, Theorem 4.23 in \cite{leyffer2022sequential}]\label{thm:pure_tr_asymptotics}
	Let $\alpha \in (0.5,1)$.
	Let \cref{ass:standing} hold. Let the iterates $(w^n)_n$ be produced by \cref{alg:trm}.
	Let $\nabla F(w^n) \in C^2(\bar{\Omega})$ for all $n \in \N$.
	Then all iterates are feasible for \eqref{eq:p} and 
	the sequence of objective values $(J_\alpha(w^n))_n$
	is monotonically decreasing.
	Moreover, one of the following mutually exclusive outcomes holds:
	\begin{enumerate}
		\item\label{itm:finite_seq_tr} The sequence $(w^n)_n$ is finite.
		The final element $w^N$ of $(w^n)_n$ solves the
		trust-region subproblem $\text{\emph{\ref{eq:tr}}}(w^N,\nabla F(w^N),\Delta)$
		for some $\Delta > 0$ and is stationary for \eqref{eq:p}.
		\item\label{itm:finite_seq_tr_contract} The sequence 
		$(w^n)_n$ is finite and the inner loop does not terminate
		for the final element $v^N$, which is stationary for \eqref{eq:p}.
		\item\label{itm:infinite_seq_sl} The sequence $(w^n)_n$ has a
		pseudo-weak-$^*$ accumulation
		point in $\calF$. Every pseudo-weak-$^*$ accumulation point of $(w^n)_n$
		is feasible, and strict. If $w$ is a pseudo-weak-$^*$ accumulation point of
		$(w^n)_n$ that satisfies $\nabla F(w) \in C^2(\bar{\Omega})$, then it is stationary
		for \eqref{eq:p}.
		
		If the trust-region radii are bounded away from zero for a subsequence $(w^{n_\ell})_\ell$, that is,
		if $0 < \underline{\Delta} \coloneqq \liminf_{n_\ell\to\infty} \min_{k}\Delta^{n_{\ell}+1,k}$
		and $\bar{w}$ is a pseudo-weak-$^*$ accumulation point of $(w^{n_\ell})_\ell$ with $\nabla F(\bar{w}) \in C(\bar{\Omega})$,
		then $\bar{w}$ solves $\text{\emph{\ref{eq:tr}}}(\bar{w},\nabla F(\bar{w}),\underline{\Delta} / 2)$.
	\end{enumerate}
\end{theorem}
\begin{proof}
As in the proofs of Theorem 6.4 in \cite{manns2023on} and Theorem 4.23 in \cite{leyffer2022sequential}
it follows that \cref{alg:trm} produces a sequence of feasible iterates $(w^n)_n$ with corresponding
montonotically decreasing sequence of objective function values $(J_\alpha(w^n))_n$. Again, as in
the proofs of Theorem 6.4 in \cite{manns2023on} and Theorem 4.23, it suffices to prove
Outcome \ref{itm:infinite_seq_sl}
in case that Outcomes \ref{itm:finite_seq_tr} and \ref{itm:finite_seq_tr_contract} do not hold.
In this argument \cref{prp:inner} takes the role of Lemma 6.2 in \cite{manns2023on} and Lemma 4.19
in \cite{leyffer2022sequential}. As in \cite{manns2023on}, we consider four steps of the proof that
Outcome \ref{itm:infinite_seq_sl} holds into four parts.
	
\textbf{Outcome \ref{itm:infinite_seq_sl} (1) existence and feasibility of pseudo-weak-$^*$
accumulation points:} This follows with the same arguments as are carried out for the existence
of minimizers in \cref{prp:existence}.

\textbf{Outcome \ref{itm:infinite_seq_sl} (2) pseudo-weak-$^*$ accumulation points
	are strict:} This follows with the same arguments as are carried out
	in the corresponding paragraph in the proof of Theorem 6.4 in \cite{manns2023on}.
	The only change is that the $\TV$-term is replaced by the $R_\alpha$.

\textbf{Outcome \ref{itm:infinite_seq_sl} (3) strict accumulation points
	are optimal for \eqref{eq:tr} if the trust-region is bounded away
	from zero:} This follows with the same arguments as are carried out
in the corresponding paragraph in the proof of Theorem 6.4 in \cite{manns2023on}
when the role of Theorem 5.2 in \cite{manns2023on} is taken by
\cref{thm:tr_gamma_convergence} and the role of Proposition 5.5 in
\cite{manns2023on} is taken by \cref{prp:minimizers_are_stationary}.

\textbf{Outcome \ref{itm:infinite_seq_sl} (4) strict accumulation points
	are stationary if the trust-region radius vanishes:}
This follows with the same arguments as are carried out
in the corresponding paragraph in the proof of Theorem 6.4 in \cite{manns2023on}
with the following adaptions. The assumed violation of \emph{L-stationarity} in
Theorem 6.4 is replaced by a violation of \eqref{eq:stationarity}
and the $\TV$-term is replaced by $R_\alpha$. The roles of
Lemmas 3.3 and 3.5 in \cite{manns2023on} are taken by
\cref{prp:variation_fractional_perimeter,prp:variation_linearized_objective}.
For the conditions (a), (b), and (c) on the choice of $\Delta^*$ in
Theorem 6.4, we replace
$\varepsilon_1$ by $\varepsilon_1^\alpha$ in (a) and the two occurrences
of $\Delta^* \kappa^{-1}$ by $(\Delta^* \kappa^{-1})^{\frac{1}{\alpha}}$
in order to account for the exponent $\alpha$ in the estimate
\cref{lem:sym_diff_boundedness}.
Note that the non-negativity in (b) can be achieved because 
of the assumption $\alpha > \frac{1}{2}$, which implies that
$(\Delta^*)^{\frac{1}{\alpha}}$ dominates $(\Delta^*)^2$ for $\Delta^* \searrow 0$.
Moreover, below (a), (b), and (c) we choose
$t$ as $\left(\tfrac{\Delta^*}{2\kappa}\right)^{\frac{1}{\alpha}}$
instead of $\tfrac{\Delta^*}{2\kappa}$. Then the remaining steps can be carried out
as in the proof of Theorem 6.4 in \cite{manns2023on}.
\end{proof}

\section{Computational experiment}\label{sec:computational}
We provide a computational example to give a qualitative impression of
the behavior of the resulting discretization.
We consider the binary control of an elliptic boundary value problem by means of
a source term that enters the right hand side of the PDE. The main objective
is a tracking-type functional so that the problems \eqref{eq:p} become
\begin{gather}\label{eq:pc}
\begin{aligned}
\minimize_{u, w} \enskip & \frac{1}{2}\|u - u_d\|_{L^2}^2
+ \eta R_\alpha(w)\\
\st\enskip & - \nu \Delta u = w\text{ in } \Omega,\quad u|_{\partial \Omega} = 0,\\
& w(x) \in W \coloneqq \{0,1\} \subset \Z \text{ a.e.\ in } \Omega.
\end{aligned}
\tag{P$_\alpha$}
\end{gather}
We choose $\Omega = (0,1)\times (0,1)$, $u_d \in L^2(\Omega)$ such that it
cannot be realized, implying the first term in the objective is bounded
below away from zero. Moreover, we choose $\nu= \tfrac{1}{25}$ and $\eta = 5\cdot 10^{-5}$.
We discretize the problem and run the trust-region algorithm for
the choices $\alpha = 0.5$, $\alpha = 0.9$, and for the limit with $R(w) = \omega_{d-1} P(E_1)$.
In all cases, we initialize the algorithm with $w^0 = 0$.

In order to discretize the problem, we choose a piecewise constant ansatz
for the control input $w$ on a uniform grid of squares of size $n \times n$
with $n = 48$. For the limit case, an isotropic discretization of the total
variation seminorm in integer optimization is computationally difficult and
recent approaches \cite{schiemann2024discretization} are not computationally
mature enough so far. Therefore, we compute the total variation as the length
of the interfaces along the boundaries of the grid cells, which implies
an anisotropic behavior in the limit. We then obtain integer linear programs for
the trust-region subproblems as derived in Appendix B of \cite{manns2023on}.
We expect that rectangular shapes are preferred for this anisotropic discretization.
For the choices $\alpha = 0.5$, $\alpha = 0.9$, we tabulate the possible
contributions of pairs of different cells to the double integral \eqref{eq:frac_perim_as_one_double_integral}
and formulate the resulting trust-region subproblems as integer linear
programs. Due to the double integral, this amounts to a number of variables
in the order of $n^4$, which is too much for standard integer programming
solvers to handle easily. To alleviate this issue, we limit the contributions that
are taken into account in the inner integral so that
only cells whose center is within an $\ell_2$-distance of at most
$7 \frac{1}{n}$ to the center of a cell in the outer integral are taken
into account in the inner integral. This is justified by the decay
of the integral kernel of the Gagliardo seminorm with distance to the current
point. For the limit case, we compute the limiting regularizer as the
sum of the interface lengths multiplied by their jump heights as in
\cite{manns2023on}.

We discretize the PDE using the open source library 
FEniCSx\footnote{\url{https://fenicsproject.org/}} 
\cite{ScroggsEtal2022,BasixJoss,AlnaesEtal2015}, where we use a much
finer mesh than $48\times 48$ cells to solve the PDE
and solve the subproblems using the integer programming solver
Gurobi\footnote{\url{https://www.gurobi.com/}} \cite{gurobi}.
We run the problems in single CPU mode on one node of TU
Dortmund's Linux HPC cluster LiDO3 (node: 
2x AMD EPYC 7542 32-Core CPUs and 1024 GB RAM).
Even with our relatively coarse discretization and the approximation of
the inner integral, many of the subproblems for $\alpha = 0.5$ and
$\alpha = 0.9$ are very expensive, a lot of time to solve and the computations
take almost three weeks (with several subproblems being solved inexactly
because they did not solve to global optimality within 30 hours).
The limit case is solved in a couple of hours.

This high computational burden shows that more work is
necessary to solve these structured integer linear programs efficiently
and approximate the involved integrals sensibly. 
For $\alpha = 0.5$, the trust-region algorithm accepts 31 steps until
the trust-region radius contracts, that is it drops below the volume
of one cell in the control grid.
For $\alpha = 0.9$, the trust-region algorithm accepts 28 steps until
the trust-region radius contracts.
For the limiting case, the trust-region algorithm accepts 36
until the trust-region radius contracts.

The limiting case shows an anisotropic behavior. This is expected because
the geometric restriction induced by the control function ansatz
implies that all interface lengths (jump height multiplied by
length) are computed along the boundaries of the discretization
into squares, see also the comments at the end of section 2
in \cite{severitt2023efficient}. 
This is alleviated for $\alpha = 0.5$ and $\alpha = 0.9$ and a less
anisotropic behavior can be observed. The three resulting controls
at the respective final iterations are shown in \cref{fig:shapes}.
\begin{figure}
	\centering
	\begin{subfigure}{.33\textwidth}
		\centering
		\includegraphics[width=\linewidth]{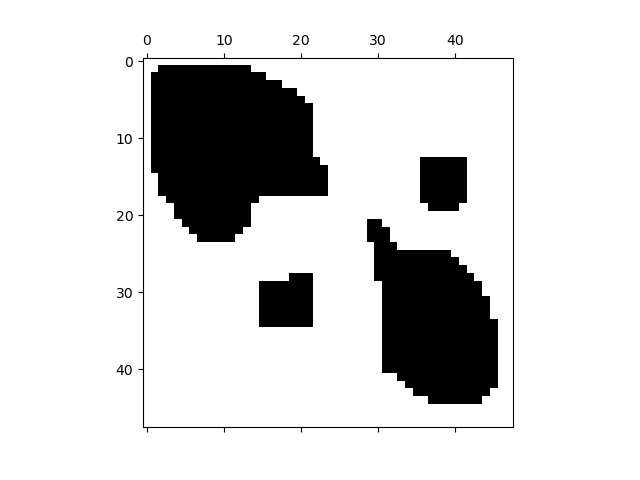}
		\caption{$\alpha = 0.5$}
	\end{subfigure}%
	\begin{subfigure}{.33\textwidth}
		\centering
		\includegraphics[width=\linewidth]{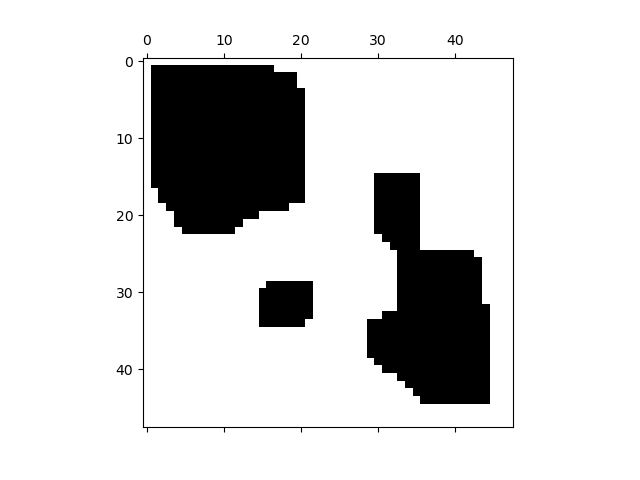}
		\caption{$\alpha = 0.9$}		
	\end{subfigure}%
	\begin{subfigure}{.33\textwidth}
		\centering
		\includegraphics[width=\linewidth]{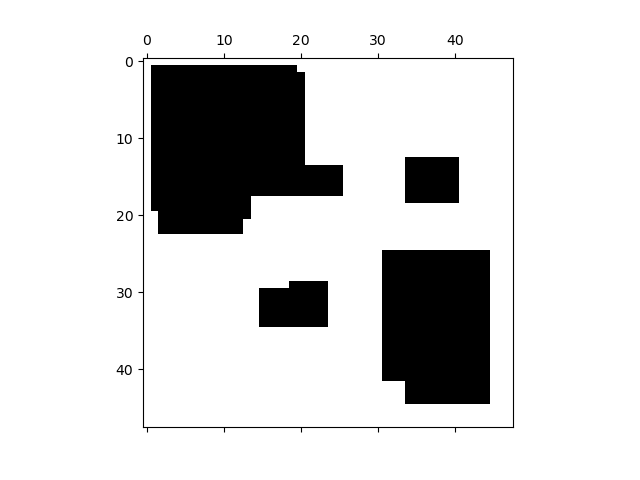}
		\caption{``$\alpha = 1$'' (perimeter)}		
	\end{subfigure}
	\caption{Resulting shapes for the different values of $\alpha$
	at the final iteration of the trust-region algorithm
 	executed using discretized subproblems.}
	\label{fig:shapes}
\end{figure}
We also note that when running the experiments again (where parallelization in
Gurobi is turned on), slightly different results are returned because the integer problems
are numerically challenging and we are at the limits of what Gurobi can handle
so that several similar integer configurations are within the tolerances and
the outcome is not deterministic. We have also observed that this leads
our trust-region algorithm to follow slightly different paths and contract
at different stationary points. This highlights even more that more work
is necessary to solve the subproblems efficiently to global optimality (or
a constant factor approximation). 

\section*{Conclusion}
Our theoretical analysis opens a sensible way of approaching the computationally 
difficult approximation of the anisotropic total variation in contexts
with discreteness restrictions on the variables and discretizations with fixed
geometries. Specifically, we have approximated the boundary integral by a double
volume integral, where the approximation properties are carried out by means of
the fractional nonlocal perimeter.
Like other recent steps in this direction, our computational experiments
show that we end up with a computationally very challenging problem, which needs to
be understood and scaled to meaningful problem sizes in the future.

\section*{Acknowledgments}
The authors gratefully acknowledge computing time on the LiDO3 HPC cluster at TU Dortmund, partially
funded in the Large-Scale Equipment Initiative by the Deutsche Forschungsgemeinschaft (DFG) as project
271512359. The authors like to thank an anonymous referee for valuable feedback on the manuscript.

\bibliography{references}{}
\bibliographystyle{abbrv} 

\section*{Statements and Declarations}
\subsection*{Funding}
Harbir Antil H. Antil is partially supported by NSF grant DMS-2110263, Air Force Office of Scientific Research (AFOSR) under Award NO: FA9550-22-1-0248, and Office of Naval Research (ONR) under Award NO: N00014-24-1-2147.
Paul Manns is partially supported by Deutsche Forschungsgemeinschaft (DFG) under
project no.\ 515118017.

\subsection*{Competing Interests}
The authors have no relevant financial or non-financial interests to disclose.

\end{document}